\documentclass[a4paper,11pt]{article}
\usepackage[english]{babel}
\usepackage{graphicx}
\usepackage{amscd}
\usepackage{amsfonts}
\usepackage{latexsym}
\usepackage{color}
\usepackage{enumerate}
\usepackage{amsthm}
\usepackage{amssymb,amsmath,mathrsfs}
\usepackage[all]{xy}
\usepackage{hyperref}
\newtheorem{thm}{Theorem}[section]
\newtheorem{lem}[thm]{Lemma}
\newtheorem{prop}[thm]{Proposition}
\newtheorem{cor}[thm]{Corollary}
\newtheorem{dfn}[thm]{Definition}
\theoremstyle{remark}
\newtheorem{remark}[thm]{Remark}

\newcommand{\abcd}[4]{\ensuremath{\begin{pmatrix}  #1&#2\\#3&#4\end{pmatrix}}}

\def\Z{{\ensuremath{\mathbb{Z}\,}}}

\def\C{{\ensuremath{\mathbb{C}}}}
\def\P{{\ensuremath{\mathbb{P}}}}

\def\Q{{\ensuremath{\mathbb{Q}}\,}}
\def\Qb{{\ensuremath{\overline{\mathbb{Q}}\,}}}
\def\O{{\ensuremath{\mathcal{O} }}}

\def\F{{\ensuremath{\mathbb{F}}}}

\def\p{{\ensuremath{\mathfrak{p}\,}}}

\DeclareMathOperator{\spl}{split}
\DeclareMathOperator{\spc}{sp.Car}

\DeclareMathOperator{\disc}{disc}

\DeclareMathOperator{\Gal}{Gal}
\DeclareMathOperator{\Frob}{Frob}
\DeclareMathOperator{\Aut}{Aut}

\DeclareMathOperator{\GL}{GL}
\DeclareMathOperator{\PGL}{PGL}
\DeclareMathOperator{\SL}{SL}
\DeclareMathOperator{\Id}{Id}

\DeclareMathOperator{\ord}{ord}
\DeclareMathOperator{\spec}{Spec}

\title{A local-global principle for isogenies of prime degree over number fields}
\author{Samuele Anni}

\begin{document}
\maketitle
\begin{abstract}
We give a description of the set of exceptional pairs for a number field $K$, that is the set of pairs $(\ell , j(E))$, where $\ell$ is a prime and $j(E)$ is the $j$-invariant of an elliptic curve $E$ over $K$ which admits an $\ell$-isogeny locally almost everywhere but not globally. We obtain an upper bound for $\ell$ in such pairs in terms of the degree and the discriminant of $K$. Moreover, we prove finiteness results about the number of exceptional pairs.
\end{abstract}

\section{Introduction}

Let $E$ be an elliptic curve over a number field $K$, and let $\ell$ be a prime number. If we know residual information on $E$, i.e.\ information over the reduction of $E$ for a set of primes with density one, can we deduce global information about $E$?

One of the first to ask this kind of question was Serge Lang: let $E$ be an elliptic curve over a number field $K$, and let $\ell$ be a prime numbe; if $E$ has non-trivial $\ell$-torsion locally at a set of primes with density one{\bf ,} then does $E$ have non-trivial $\ell$-torsion over $K$? In 1981, Katz studied this local-global principle, see \cite{k}. He was able to prove that if $E$ has non-trivial $\ell$-torsion locally at a set of primes with density one, then there exists a $K$-isogenous elliptic curve which has non-trivial $\ell$-torsion over $K$. He proved this by reducing the problem to a purely group-theoretic statement. Moreover, he was able to extend the result even further to $2$-dimensional abelian varieties and to give a family of counterexamples in dimension $3$.

Here we consider the following variation on this question: let $E$ be an elliptic curve over a number field $K$, and let $\ell$ be a prime number; if $E$ admits an $\ell$-isogeny locally at a set of primes with density one then does $E$ admit an $\ell$-isogeny over $K$? 

Recently, Sutherland has studied this problem, see \cite{suth}. 
Let us recall that, except in the case when the $j$-invariant is $0$ or $1728$, whether an elliptic curve admits an $\ell$-isogeny over $K$ or not depends only on its $j$-invariant. As in \cite{suth}, we will only consider elliptic curves with $j$-invariant different from $0$ and $1728$.  

\begin{dfn} 
Let $K$ be a number field and $E$ an elliptic curve over $K$. A pair $(\ell , j(E)) $ is said to be \emph{exceptional} for $K$ if $E/K$ admits an  $\ell$-isogeny locally almost everywhere, i.e.\ for a set of primes of $K$ of density one, but not over $K$.
\end{dfn}
For primes of good reduction and not dividing $\ell$, the definition of local isogeny is the natural one , and it is recalled in Definition~\ref{localisogeny} below. 

Let us remark that if $( \ell , j(E))$ is an exceptional pair for the number field $K$, then any $E_D$, quadratic twist of $E$, gives rise to the same exceptional pair. Indeed, the Galois representation associated to the $\ell$-torsion of $E$ and the one of $E_D$ are twists of each other: 
$\rho_{E_D,\ell}\simeq \chi_D \otimes \rho_{E,\ell}$, where $\chi_D$ is a quadratic character. Hence the projective images of such representations are isomorphic.

A curve occurring in an exceptional pair will admit an $\ell$-isogeny globally over a small extension of the base field: more precisely, we can state the following Proposition, which is a sharpened version of  a result of Sutherland (for a proof see Section $3$ herein):
\begin{prop}
Let $E$ be an elliptic curve defined over a number field $K$, let $\ell$ be an odd prime number and assume that $\sqrt{\left(\frac{-1}{\ell}\right)\ell}$ does not belong to $ K$. 
Suppose that $E/K$ admits an $\ell$-isogeny locally at a set of primes with density one. Then $E$ admits an $\ell$-isogeny over $K(\sqrt{-\ell})$.
Moreover, if $\ell=2,3$ or $\ell \equiv 1\bmod 4$ then $E$ admits an $\ell$-isogeny over $K$.
\label{1.2}
\end{prop}
There are examples, for $\ell \equiv 3\bmod 4$ and $\ell \geq 7$, in which it is necessary to extend the base field to have a global isogeny. In particular, Sutherland proved that over $\Q$ the following optimal result holds: 

\begin{thm}[{\cite[Theorem 2]{suth}}]
The only exceptional pair for $\Q$ is $(7, 2268945/128)$.
\end{thm}

This theorem is proved by applying \cite[Theorem~1.1]{Parent},  which  asserts that for all primes $\ell \equiv 3 \bmod 4$, with $\ell  > 7$,  the only rational non-cuspidal points on $X_{\spl}(\ell)(\Q)$ correspond to elliptic curves with complex multiplication (for a definition of this modular curve see Section $5$ herein). 
Hence, over $\Q$ there exists only one counterexample to the local-global principle for $7$-isogenies and there is none for $\ell$-isogenies for $\ell>7$. 

We will prove that a similar dichotomy is true for any number field: the number of counterexamples to the local-global principle about $\ell$-isogenies for $\ell>7$ is always finite and the number of counterexamples to the local-global principle about $7$-isogenies and $5$-isogenies may or may not be finite, depending in one case on the rank of a given  elliptic curve, and  in the other on a  condition on the number field.

Namely, the main result of this paper is the following:
\theoremstyle{plain}
\newtheorem*{t1}{Main Theorem}
\begin{t1}
Let $K$ be a number field of degree $d$ over $\Q$ and discriminant $\Delta$, and let $\ell_K:=\max\left\{|\Delta|, 6d{+}1\right\}$. The following holds: 
\begin{enumerate}[$(1)$]
\setlength{\parskip}{0pt}
\setlength{\parsep}{0pt}
	\item if $( \ell , j(E))$ is an exceptional pair for the number field $K$ then $\ell \leq \ell_K$;
	\item there are only finitely many exceptional pairs for $K$ with $7< \ell \leq \ell_K$;
	\item the number of exceptional pairs for $K$ with $\ell=7$ is finite or infinite,  depending whether the rank over $K$ of the Elkies\--Sutherland's elliptic curve:
    \[y^2=x^3-1715x+33614\] is zero or non\--zero, respectively;

	\item there exist no exceptional pairs for $K$ with $\ell=2$ or with $\ell=3$; 
	\item there exist exceptional pairs for $K$ with $\ell=5$ if and only if $\sqrt{5}$ belongs to $K$. Moreover, if $\sqrt{5}$ belongs to $K$ then there are infinitely many exceptional pairs for $K$ with $\ell=5$.
\end{enumerate}
\end{t1}

Actually, we prove more precise results that will be discussed in the following sections. Before entering the details of the proofs, let us give a rough idea of our strategy for the point $(1)$ above. 

The pair~$( \ell , j(E))$ is exceptional for a number field~$K$ if and only if the action of $G\subseteq \GL_2(\F_\ell)$, the image of the Galois representation associated to the $\ell$-torsion of $E$, on~$\P(E[\ell ])\simeq \P^1 (\F_\ell )$ has no fixed point, whereas every $g\in G$ leaves a line stable, that is, all~$g\in G$ have a reducible characteristic polynomial. Using Dickson's classification of subgroups of~$\PGL_2 (\F_\ell )$, one sees that, up to conjugation, $G$ has to be either the inverse image of an exceptional group (but this case is known to happen only for small~$\ell$) or contained in the normalizer of a split Cartan subgroup, that is $G=\left\{\begin{pmatrix}a&0\\0&b\end{pmatrix},\begin{pmatrix}0&a\\b&0\end{pmatrix}| a, b\in \Gamma\right\}$ for $\Gamma$ a subgroup of $\F_\ell^*$.

In the last case, using notations as in \cite[Section $2$]{Parent} or in Section $5$ herein, exceptional pairs therefore induce points in $X_{\mathrm{split}} (\ell )(K)$ not lifting to~$X_{\mathrm{sp. Car}} (\ell )(K)$ (forgetting cases corresponding to exceptional groups). The existence of such points is in general a wide open question; however, in our case, the reducibility of the characteristic polynomials implies that   
$\Gamma$ is actually a subgroup of squares: $\Gamma \subseteq (\F_\ell^* )^2$. 
By the property of the Weil pairing, one deduces that~$K(\sqrt{\left(\frac{-1}{\ell}\right)\ell} )\subseteq L$, the field of definition of the $\ell$-isogeny. From this, in the case 
$\sqrt{\left(\frac{-1}{\ell}\right)\ell}$ does not belong to $K$, we conclude that the well-known shape of inertia at $\ell$ inside~$G$ gives a contradiction for~$\ell > 6[K:\Q ]{+}1$.
 
This article is organized as follows.
For the convenience of the reader, we recall in Section $2$ the results obtained by Sutherland in \cite[Section $2$]{suth}. 
In Section $3$, we study exceptional pairs over arbitrary number fields. First we deduce the effective version of Sutherland's result, then we describe how to tackle the case not treated by Sutherland and finally we prove statement $(4)$ of the Main Theorem (see~Proposition~\ref{small}). 
In Section $4$, we prove, as we already commented, the bound given in $(1)$ of the Main Theorem (Corollary~\ref{maincor}). In Section~$5$ we discuss finiteness results for the set of exceptional pairs and we prove statements $(2)$ (Theorem~\ref{fin}), $(5)$ (Corollary~\ref{five}) and $(3)$ (Proposition~\ref{7}) of the Main Theorem. Finally, in Section $6$, we give conditions under which an exceptional pair does not have complex multiplication.

After this article was written, Banwait and Cremona have presented a detailed description of the local-global principle about $\ell$-isogenies for quadratic number fields, see \cite{barcre}. 

\section{Sutherland's results}
Let us recall the definition of local $\ell$-isogeny for an elliptic curve:
\begin{dfn}\label{localisogeny}
Let $E$ be an elliptic curve over a number field $K$, let $\ell$ be a prime number. If $\p$ is a prime of $K$ where $E$ has good reduction, $\p$ not dividing $\ell$, we say that $E$ admits an $\ell$-isogeny \emph{locally} at $\p$ if the reduction of $E$ modulo $\p$ admits an $\ell$-isogeny defined over the residue field at $\p$.
\end{dfn}

Let us remark that for a prime  $\p$ of $K$ where $E$ has good reduction, $\p$ not dividing $\ell$, the definition given is equivalent to saying that the N\'eron model of $E$ over the ring of integers of $K_\p$ admits an $\ell$-isogeny. This follows essentially because the $\ell$-isogeny in this case is \'etale.

Sutherland has proved, under certain conditions, that for an elliptic curve defined over a number field, the existence of local $\ell$-isogenies for a set of primes with density one implies the existence of a global $\ell$-isogeny: 
\begin{thm}[{\cite[Theorem~$1$]{suth}}] \label{suth}Let $E$ be an elliptic curve over a number field $K$ with 
$j(E)\notin\{0, 1728\}$, and let $\ell$ be a prime number. Assume that $\sqrt{\left(\frac{-1}{\ell}\right)\ell}\notin K$, and suppose that  $E/K$ admits an $\ell$-isogeny locally at a set of primes with density one. Then $E$ admits an $\ell$-isogeny over a {quadratic extension of $K$}. Moreover, if $\ell \equiv 1\bmod 4$ or $\ell < 7$, $E$ admits an $\ell$-isogeny defined over $K$.
\end{thm}

Let us recall briefly how this theorem is proved. The main tool used is the theory of Galois representations attached to elliptic curves, see \cite{serre72}, to reduce the problem to a question regarding subgroups of $\GL_2(\F_\ell)$.

There is a natural action of $\GL_2(\F_\ell)$ on $\mathbb{P}^1(\F_\ell)$, and the induced action of $\PGL_2(\F_\ell)$ is faithful. For an element $g$ of $\GL_2(\F_\ell)$ or of $\PGL_2(\F_\ell)$, we will denote, respectively, by $\mathbb{P}^1(\F_\ell){/}g$ the set of $g$-orbits of $\mathbb{P}^1(\F_\ell)$ and by $\mathbb{P}^1(\F_\ell)^g$ the set of elements fixed by $g$. 

\begin{lem}[{\cite[Lemma $1$]{suth}}]\label{comb}
Let $G$ be a subgroup of $ \GL_2(\F_\ell)$ whose image $H$ in $\PGL_2(\F_\ell)$ is not contained in $ \SL_2(\F_\ell)/\left\{\pm 1\right\}$. Suppose $|\mathbb{P}^1(\F_\ell)^g | > 0$  for all $g \in G$  but $|\mathbb{P}^1(\F_\ell)^G | = 0$.\\
Then $\ell \equiv 3 \bmod 4$ and the following holds:
\begin{enumerate}[$(1)$]
\item  $H$ is dihedral of order $2n$, where $n>1$ is an odd divisor of $ (\ell{-}1)/2$;
\item $G$ is properly contained in the normalizer of a split Cartan subgroup;
\item $\mathbb{P}^1(\F_\ell){/}G$,  the set of $G$-orbits of $\mathbb{P}^1(\F_\ell)$, contains an orbit of size $2$.
\end{enumerate}
\end{lem}
This result is an application of the orbit-counting lemma combined with Dickson's classification of subgroups of $\PGL_2(\F_\ell)$, see \cite{di} or \cite{lang}, and it is one of the key steps in the proof of Theorem~\ref{suth}. Let us notice that if the hypotheses of Lemma~\ref{comb} are satisfied, then it follows that  $\ell \neq 3$ because $n>1$ in $(1)$.

\begin{remark}
Given an elliptic curve $E$ over a number field $K$, the compatibility between $\rho_{E,\ell}$, the Galois representation associated to the $\ell$-torsion group $E[\ell ]$, and the Weil pairing on $E[\ell ]$ implies that for every $\sigma \in \Gal(\Qb/K)$ we have $\sigma(\zeta_\ell)=\zeta_\ell^{\det(\rho_{E,\ell}(\sigma))}$, where $\zeta_\ell$ is an $\ell$th-root of unity. Let us assume that $\ell$ is an odd prime. Therefore, $\zeta_\ell$ is in $K$ if and only if $G=\rho_{E,\ell}(\Gal(\Qb/K))$ is contained in $\SL_2(\F_\ell)$. 
Moreover, using the Gauss sum: $\sum_{n=0}^{\ell-1}\zeta_\ell^{n^2}=\sqrt{\left(\frac{-1}{\ell}\right)\ell}$ and denoting by $H$ the image of $G$ in $\PGL_2(\F_\ell)$, it follows that $H$ is contained in $ \SL_2(\F_\ell)/\left\{\pm 1\right\}$ if and only if $\sqrt{\left(\frac{-1}{\ell}\right)\ell}$ is in $K$. 
\label{sqr}
\end{remark}

\begin{proof}[of Theorem~\ref{suth}.] For every  $g \in G=\rho_{E,\ell}(\Gal(\Qb/K))$, it follows from Chebotarev density theorem that there exists $\p\subset \O_K$, a prime in the ring of integers of $K$, such that $g=\rho_{E,\ell}(\Frob_\p)$ and $E$ admits an $\ell$-isogeny locally at $\p$. 

The Frobenius endomorphism fixes a line in  $E\left[\ell\right]$, hence $|\mathbb{P}^1(\F_\ell)^g| > 0$ for all $g \in G$. If $|\mathbb{P}^1(\F_\ell)^G | > 0$, then $\Gal(\Qb/K) $ fixes a linear subspace of  $E\left[\ell\right]$ which is the kernel of an  $\ell$-isogeny defined over $K$.
 
Hence, let us assume  $|\mathbb{P}^1(\F_\ell)^G | =0$.  No subgroup of $\GL_2(\F_2)$ satisfies $|\mathbb{P}^1(\F_2)^G | =0$ and $|\mathbb{P}^1(\F_2)^g | >0$ for all $g\in G$, so $\ell$ is odd.
The hypotheses on $K$ combined with the Weil pairing on the $\ell$-torsion implies that some element of $G$ has a non-square determinant, hence the image of $G$ in $\PGL_2(\F_\ell )$ does not lie in $\SL_2(\F_\ell)/\left\{\pm 1\right\}$. The hypotheses of Lemma~\ref{comb} are satisfied, thus  $\ell \equiv 3 \bmod 4$,  the prime $\ell$ is different from $3$, and $\mathbb{P}^1(\F_\ell){/}G$ contains an orbit of size $2$. Let $x \in \mathbb{P}^1(\F_\ell)$ be an element of this orbit, its stabilizer is a subgroup of index $2$. By Galois theory, it corresponds to a quadratic extension of $K$ over which $E$ admits an isogeny of degree $\ell$ (actually, two such isogenies). \end{proof} 

\begin{remark}
Since  no subgroup of $\GL_2(\F_2)$ satisfies $|\mathbb{P}^1(\F_2)^G | =0$ and $|\mathbb{P}^1(\F_2)^g | >0$ for all $g\in G$, then there is no exceptional pair with $\ell=2$.
\label{2out}
\end{remark}

\section{Exceptional pairs and Galois representations}
The study of the local-global principle about $\ell$-isogenies over an arbitrary number field $K$ depends on whether $\sqrt{\left(\frac{-1}{\ell}\right)\ell}$ belongs  to $K$ or not. By Remark~\ref{2out}, it follows that there is no exceptional pair with $\ell=2$, hence, unless otherwise stated, we will denote by $\ell$ an odd prime.

First, let us assume that $\sqrt{\left(\frac{-1}{\ell}\right)\ell}$ does not belong to $K$. Theorem~\ref{suth} implies that an exceptional pair for $K$ is no longer exceptional for a quadratic extension of $K$: in this section we will describe this extension.

\begin{prop} \label{ima}
Let $(\ell, j(E))$ be an exceptional pair for the number field $K$, not containing $\sqrt{\left(\frac{-1}{\ell}\right)\ell}$, with $j(E)$ not in $\left\{0,1728\right\}$. Let $G$ be the image of $\rho_{E,\ell}$ 
and let $H$ be its image in $\PGL_2(\F_\ell)$. 
Let $\mathscr{C}\subset G$ be the preimage of the maximal cyclic subgroup of $H$. Then $\det(\mathscr{C})\subseteq (\F_\ell^\ast)^2$, where $(\F_\ell^\ast)^2$ denotes the group of squares in $\F_\ell^\ast$.

\end{prop}
\begin{proof} Let $(\ell, j(E))$ be an exceptional pair for the number field $K$, and assume that $\sqrt{\left(\frac{-1}{\ell}\right)\ell}\notin K$. Therefore,  $\ell$ is odd by Remark~\ref{2out} and Remark~\ref{sqr} implies that $H$ is not contained in $\SL_2(\F_\ell)/\left\{\pm 1\right\}$. Hence, applying Lemma~\ref{comb}, we have that $\ell \equiv 3 \bmod 4$, up to conjugation, $H$ is a dihedral group of order $2n$, where $n>1$ is an odd divisor of $(\ell{-}1)/2$ and, moreover, $G$ is properly contained in the normalizer of a split Cartan subgroup. In particular, we have that $(n, \ell{+}1)=1$ and $(n,\ell)=1$, because $n$ is odd and it is a divisor of $\ell{-}1$.
Let us underline that,  since $n$ is odd the maximal cyclic subgroup inside $H$ is determined uniquely up to conjugation: it is the unique subgroup of index $2$.

 Let $A\in G \subset \GL_2(\F_\ell)$ be a preimage of some generator of the maximal cyclic subgroup inside $H$. Let us prove that $A$ is conjugate in  $\GL_2(\F_\ell)$ to a matrix of the type $\abcd{\alpha}{0}{0}{\beta}$, with $\alpha, \beta\in \F_\ell^\ast$ and $\alpha/\beta$ of order $n$ in $\F_\ell^\ast$.

Extending the scalars to $\F_{\ell^2}$ if necessary, we can put $A$ in its Jordan normal form. Then either $A\cong\abcd{\alpha}{0}{0}{\beta}$ with 
$\alpha,\beta\in \F_{\ell^2}$, or $A\cong\abcd{\alpha}{1}{0}{\alpha}$. Since we have $A^n=\lambda \cdot \Id$, for $\lambda\in \F_\ell^\ast$, and $(n,\ell)=1$ then the second case cannot occur. 
We claim that $\alpha,\beta\in \F_{\ell}^\ast$. In fact, let us proceed by contradiction: if $\alpha,\beta\in \F_{\ell^2}\mathcal{n}\F_\ell$ then $\beta=\overline{\alpha}$, the conjugate of $\alpha$ over $\F_\ell$. This means that $\mathbb{P}^1(\F_\ell)^A$ is empty because $A$ has no eigenvalues in $\F_\ell$ and this is not possible because $E$ admits an $\ell$-isogeny locally everywhere and by Chebotarev density theorem $A=\rho_{E,\ell}(\Frob_\p)$ with $\p\subset \O_K$ prime. Hence $\alpha,\beta\in \F_{\ell}^\ast$. 
Let us write $\alpha =\mu^i$ and $\beta =\mu^j$ for some generator $\mu$ of $\F_\ell^*$. Then $\mu^{in} =\alpha^n =\beta^n =\mu^{jn}$ so that 
$\mu^{n(i-j)} =1$ and $n(j-i) \equiv 0$ mod $\ell{-}1$. As $n$ is odd, $(j-i)$ has to be even, hence so is $(i+j)$. Therefore, 
$\det(A)=\alpha \beta =\mu^{i+j}$ is a square in $\F_\ell^*$. Moreover, since $A$ is a preimage of some generator of the maximal cyclic subgroup inside $H$, then $\alpha/\beta$ must have order $n$.
\end{proof}

\begin{remark}
Let $(\ell, j(E))$ be an exceptional pair for the number field $K$. Let us suppose that $j(E)$ is different from $0$ and $1728$ and that $\sqrt{\left(\frac{-1}{\ell}\right)\ell}$ is not in $K$.
Since the projective image of the Galois representation associated to $E$ is dihedral of order $2n$, with $n$ an odd divisor of $(\ell{-}1)/2$, the order of $G$, the image of the Galois representation, satisfies:
\[
|G|\;\mid\; \left((\ell{-}1)\cdot 2n\right)\; \mid\; \left((\ell{-}1)\cdot \frac{(\ell{-}1)}{2} \cdot 2\right) = (\ell{-}1)^2.
\]
\label{dim}
\end{remark}

\begin{prop}\label{field}
Let $(\ell , j(E))$ be an exceptional pair for the number field $K$ and assume that $\sqrt{\left(\frac{-1}{\ell}\right)\ell}$ does not belong to $K$ and $j(E)$  is different from $0$ and $1728$ . Then $E$ admits an $\ell$-isogeny over $K(\sqrt{-\ell})$ (and actually, two such isogenies).
\end{prop}
\begin{proof}  Theorem~\ref{suth} implies that $\ell\equiv 3$ $\bmod$ $4$ and $\ell\geq 7$ since $(\ell , j(E))$ is an exceptional pair. Since $\sqrt{- \ell}\notin K$, then also $\zeta_\ell$, the $\ell$-th root of unity, is not in $K$.
Since $\ell \equiv 3\bmod 4$, the quadratic subfield of $\Q(\zeta_\ell)$ is $\Q(\sqrt{- \ell})$. \\
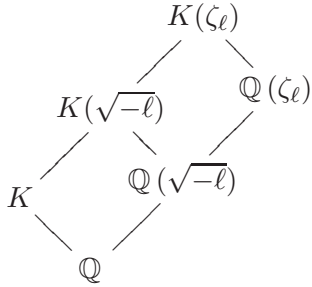
\begin{figure}[h]
\begin{minipage}{0.5\linewidth}
\begin{equation*}
\xymatrix@dr@C=1pc{
K(\zeta_\ell) \ar@{-}[r] \ar@{-}[d]& \Q(\zeta_\ell)\ar@{-}[d]\\ 
K(\sqrt{- \ell}) \ar@{-}[r] \ar@{-}[d]& \Q(\sqrt{- \ell})\ar@{-}[d]\\ 
K \ar@{-}[r]& \Q}
\end{equation*}
\end{minipage}
\begin{minipage}{0.5\linewidth} 
In particular, $\Gal(K(\zeta_\ell)/K(\sqrt{- \ell}))$ is contained in $ {\left(\F_\ell^\ast\right)^2}$, the subgroup of squares inside $\F_\ell^*$. If $\ell$ does not divide the discriminant of $K$ then the previous inclusion is an equality, since the ramifications are disjoint. Let, as before, $\rho_{E,\ell}\colon \Gal(\overline{\Q}/K)\to \GL_2(\F_\ell)$ be the Galois representation associated to $E$ and let $G$ be its image.
\end{minipage}
\end{figure}

It follows, using notation of Proposition~\ref{ima}, that $E$ admits an isogeny (actually two) on the quadratic extension $L/K$ corresponding to the Cartan subgroup $\mathscr{C}$ which is the subgroup of diagonal matrices inside $G$. By Proposition~\ref{ima}, the elements of $\mathscr{C}$ have square determinants.

On the other hand, from the properties of the Weil pairing on the $\ell$-torsion, for all $\sigma \in \Gal(\Qb/K)$ we have that 
$\det(\rho_{E,\ell}(\sigma)) = \chi_\ell(\sigma)$, where $\chi_\ell$ is the mod $\ell$ cyclotomic character. Hence, $\mathscr{C}$ is the kernel of the character $\varphi \colon G \to {\F_\ell^\ast}/{\left(\F_\ell^\ast\right)^2}\cong \left\{\pm 1 \right\}$ which makes the following diagram commute:
\[\xymatrix{
\Gal(\overline{\Q}/K)\; \ar@{>>}[r]^-{\rho_{E,\ell}} & G \ar[rr]^\varphi \ar@{>>}[d]_{\det} \ar[rd]^{\chi_\ell} & \quad & \F_\ell^\ast/\left(\F_\ell^\ast\right)^2\cong\left\{\pm 1\right\} \\
\quad &  \Gal(K(\zeta_\ell)/K) \ar@{^{(}->}[r] &  \F_\ell^\ast \cong \Aut(\mu_\ell) \ar@{>>}[ur]& \quad}\]
The character of $\Gal(\overline{\Q}/K)$ induced by $\varphi$ is not trivial because there is an element in $G$ with non square determinant since 
$\sqrt{- \ell}$ is not in $K$.  By Galois theory, the kernel of $\varphi$ corresponds to a quadratic extension of $K$, which contains $\sqrt{-\ell}$ by construction.
This implies that the extension over which $E$ admits an $\ell$-isogeny is $K(\sqrt{-\ell})$.
\end{proof}

Combining Proposition~\ref{field} and Theorem~\ref{suth}, we have proved the following result (which is Proposition~\ref{1.2} of the Introduction):
\begin{prop}\label{1.1}
Let $E$ be an elliptic curve defined over a number field $K$, with $j(E)\notin \{0, 1728\}$. Let $\ell$ be a prime number and let  $\sqrt{\left(\frac{-1}{\ell}\right)\ell}\notin K$. 
Suppose that $E/K$ admits an $\ell$-isogeny locally at a set of primes with density one, then $E$ admits an $\ell$-isogeny over $K(\sqrt{-\ell})$.
Moreover, if $\ell=2,3$  or $\ell \equiv 1\bmod 4$ then $E$ admits an $\ell$-isogeny over $K$.
\end{prop}

Now let us assume that  $\sqrt{\left(\frac{-1}{\ell}\right)\ell}$ belongs to $K$. 

As in the previous case, in order to analyse the local\--global principle, we study subgroups of $\SL_2(\F_\ell)/\left\{\pm 1\right\}$ which do not fix any point of $\mathbb{P}^1(\F_\ell)$ but whose elements do fix points.  The following lemma is a variation on the result, due to Sutherland, that  we stated in the present article as Lemma~\ref{comb}, see \cite[Lemma $1$ and Proposition $2$]{suth}. This lemma will be the key in understanding the case in which $\sqrt{\left(\frac{-1}{\ell}\right)\ell}$ belongs to $K$. We will denote by $S_n$ (resp. $A_n$) the symmetric (resp. alternating) group on $n$ elements. 

\begin{lem}\label{comb1}
Let $G$ be a subgroup of $\GL_2(\F_\ell)$ whose image $H$ in $\PGL_2(\F_\ell)$ is contained in $\SL_2(\F_\ell)/\left\{\pm 1\right\}$. 
Suppose $|\mathbb{P}^1(\F_\ell)^g | > 0$  for all $g \in G$  but $|\mathbb{P}^1(\F_\ell)^G | = 0$. Then $\ell \equiv 1 \bmod 4$ and one of the followings holds:
\begin{enumerate}[$(1)$]
\item $H$ is dihedral of order $2n$, where $n\in \Z_{>1}$ is a divisor of $\ell{-}1$; 
\item $H$ is isomorphic to one of the following exceptional groups: $A_4$, $S_4$ or $A_5$.
\end{enumerate}
\end{lem}
\begin{proof}
No subgroup of $\GL_2(\F_2)$  satisfies the hypotheses of the lemma, so we assume $\ell > 2$. The orbit-counting lemma yields:
\[|\mathbb{P}^1(\F_\ell){/}H|=\frac{1}{|H|}\sum_{h\in H}|\mathbb{P}^1(\F_\ell)^h|\geq \frac{1}{|H|}(\ell{+}|H|) >1\]
since for all $h \in H$ we have $|\mathbb{P}^1(\F_\ell)^h|> 0$  and $|\mathbb{P}^1(\F_\ell)^h|=(\ell{+}1)$ when $h$ is the identity. 

If $\ell$ divides the order of $H$, then $H$ contains an element $h$ of order $\ell$ and $\mathbb{P}^1(\F_\ell){/}h$ consists of two orbits, of sizes $1$ and $\ell$, therefore a fortiori 
$(1< )|\mathbb{P}^1(\F_\ell){/}H| \leq 2$. But this contradicts the assumption $|\mathbb{P}^1(\F_\ell)^H| = 0$. Hence $\ell$ does not divide $|H|$. 

By Dickson's classification of subgroups of $\PGL_2(\F_\ell)$ it follows that $H$ can be either cyclic, or dihedral or isomorphic to one the following groups: $S_4$, $A_4$, $A_5$. To exclude the cyclic case, suppose for the sake of contradiction that 
$H=\langle h\rangle$. This implies that $\mathbb{P}^1(\F_\ell)^h=\mathbb{P}^1(\F_\ell)^H$ and since $|\mathbb{P}^1(\F_\ell)^h|>0$ by hypothesis, we have a contradiction with 
$|\mathbb{P}^1(\F_\ell)^H|=0$. Hence $H$ is either dihedral, or isomorphic to $S_4$, or to $A_4$, or to $A_5$.

By \cite[Proposition $2$]{suth}, since $H$ is contained in $\SL_2(\F_\ell)/\left\{\pm 1\right\}$, the size of the set of $h$-orbits of $\mathbb{P}^1(\F_\ell)$ is even for each $h\in H$. Moreover, as $|\mathbb{P}^1(\F_\ell)^h |>1$, all $h$ are diagonalizable on $\F_\ell$. Then, applying the orbit-counting lemma, we have 
\begin{eqnarray}
|\mathbb{P}^1(\F_\ell){/}h|&=&\frac{1}{\ord(h)}\sum_{h'\in \left\langle h\right\rangle}|\mathbb{P}^1(\F_\ell)^{h'}|=\nonumber\\
&=& \frac{1}{\ord(h)}((\ord(h){-}1)2{+}\ell{+}1)=2+\frac{\ell{-}1}{\ord(h)}
\label{oc}
\end{eqnarray}
where $\left\langle h\right\rangle$  denotes the cyclic subgroup of $H$ generated by $h$. 
In particular, for elements of order $2$ this implies that $\ell \equiv 1 \bmod 4$.

Let us suppose that $H$ is a dihedral group of order $2n$. Then there exists $h\in H$ of order $n$, so equation (\ref{oc}) implies that $n$ divides $\ell{-}1$.
\end{proof}

Note that, in the course of the above proof, we have shown the following:

\begin{cor}\label{spC}
Let $G$ be a subgroup of $\GL_2(\F_\ell)$ whose image $H$ in $\PGL_2(\F_\ell)$ is contained in $ \SL_2(\F_\ell)/\left\{\pm 1\right\}$. 
Suppose $|\mathbb{P}^1(\F_\ell)^g | > 0$  for all $g \in G$  but $|\mathbb{P}^1(\F_\ell)^G | = 0$. If  $H$ is dihedral of order $2n$, where $n\in \Z_{>1}$ is a divisor of $\ell{-}1$, then $G$ is properly contained in the normalizer of a split Cartan subgroup and $\mathbb{P}^1(\F_\ell){/}G$ contains an orbit of size $2$.
\end{cor}

Let us now focus on case $(2)$ of Lemma~\ref{comb1}:

\begin{cor}\label{a4a5}Let $G$ be a subgroup of $\GL_2(\F_\ell)$ whose image $H$ in $\PGL_2(\F_\ell)$ is contained in $ \SL_2(\F_\ell)/\left\{\pm 1\right\}$. 
Suppose $|\mathbb{P}^1(\F_\ell)^g | > 0$  for all $g \in G$  but $|\mathbb{P}^1(\F_\ell)^G | = 0$. Then:
\begin{itemize}
	\item if $H$ is isomorphic to $A_4$ then $\ell\equiv 1 \bmod 12$;
	\item if $H$ is isomorphic to $S_4$ then $\ell\equiv 1 \bmod 24$;
	\item if $H$ is isomorphic to $A_5$ then $\ell\equiv 1 \bmod 60$.
\end{itemize}
\end{cor}
\begin{proof} 
This is an application of the orbit-counting lemma. In $A_4$ there are elements of order $2$ and $3$, and we have $\ell{>}3$ since $\GL_2 (\F_2 )\simeq S_3$. 
The equation (\ref{oc}) for elements of order $3$ implies that $\ell{-}1$ is divisible by $3$, so, since $\ell \equiv 1 \bmod 4$,  then $\ell \equiv 1 \bmod 12$. 

Applying \cite[Proposition $2$]{suth}, we see that the parity of the values of equation (\ref{oc}) determines the sign as permutation of any element of $\PGL_2(\F_\ell)$. Since $H$ is contained in $ \SL_2(\F_\ell)/\left\{\pm 1\right\}$, the value of equation (\ref{oc}) has to be even for every $h$ in $H$.
If $H$ is isomorphic to $S_4$, then it contains elements of order $4$ and this implies that $\ell{-}1$ is divisible by $8$. Repeating the argument for elements of order $3$ we conclude that $\ell\equiv 1 \bmod 24$. 
 
Analogously, if $H$ is isomorphic to $A_5$ then $\ell{>}5$: not all matrices in $\SL_2(\F_5)/\left\{\pm 1\right\}\simeq A_5$ leave a line stable. There are elements of order $3$ and $5$, then $\ell{-}1$ is divisible by $3$ and $5$. So, since $\ell \equiv 1 \bmod 4$, we have $\ell \equiv 1 \bmod 60$. 
\end{proof}

\begin{prop}\label{1.11}
Let $E$ be an elliptic curve over a number field $K$ of degree $d$ over $\Q$, with $j(E)\notin \{0, 1728\}$, and let $\ell$ be a prime number. Let us suppose 
$\sqrt{\left(\frac{-1}{\ell}\right)\ell}\in K$ and that $E/K$ admits an $\ell$-isogeny locally at a set of primes with density one. Then:
\begin{enumerate}[$(1)$]
\item if $\ell{\equiv}3\bmod 4$ the elliptic curve $E$ admits an $\ell$-isogeny over $K$;
\item if $\ell{\equiv}1\bmod 4$ the elliptic curve $E$ admits an $\ell$-isogeny over a finite extension of $K$, which can ramify only at primes dividing the conductor of $E$ and $\ell$. Moreover, if  $\ell{\equiv}{-}1\bmod 3$ or if $\ell\geq 60d{+}1$, then $E$ admits an $\ell$-isogeny over a quadratic extension of $K$.
\end{enumerate}

\end{prop}
\begin{proof} Since $\sqrt{\left(\frac{-1}{\ell}\right)\ell}$ is contained in $K$, the projective image $H$ of the Galois representation $\rho_{E,\ell}$ is contained in $\SL_2(\F_\ell)/\left\{\pm 1\right\}$, as discussed in Remark~\ref{sqr}, so we can apply Lemma~\ref{comb1}. This means that if the pair 
$(\ell, j(E))$ is an exceptional pair then $\ell{\equiv}1\bmod 4$ and $H$ has to be either a dihedral group of order $2n$, with $n$ dividing $\ell{-}1$, or an exceptional subgroup. 

If the elliptic curve $E$ admits an $\ell$-isogeny over a number field $L/K$ then there exists a one dimensional $\Gal(\Qb/L)$-stable subspace of $E[\ell]$. This subspace corresponds to a subgroup of the image of the Galois representation. In particular, the extension $L$ of $K$ over which the isogeny is defined can only ramify at primes where the representation is ramified, that is, only at primes of $K$ dividing the conductor of $E$ or $\ell$. 

If $\ell\geq 60 d{+}1$ then $H$ cannot be isomorphic to $A_4$, or to $S_4$ or $A_5$; for a proof of this fact see \cite[p.$36$]{maz}. Analogously, by Corollary~\ref{a4a5}, exceptional images cannot occur if $\ell \equiv {-}1\bmod 3$. Hence, by Corollary~\ref{spC}, in these cases the image of the Galois representation associated to $E$ is conjugated to the normalizer of a split Cartan subgroup which contains the Cartan subgroup itself with index $2$. By Galois theory, then $E$ admits an isogeny over a quadratic extension of $K$. 
\end{proof} 

Let us now describe all the possibilities that can occur at $2$, $3$ and  $5$:
\begin{prop}\label{small}
Let $K$ be a number field. There exists no exceptional pair for $K$ with $\ell=2, 3$. If $\sqrt{5}$ belongs to $K$ then there exist exceptional pairs $(5,j(E))$ for $K$ and, moreover, for such a pair, the projective image of $\rho_{E,5} (\Gal(\Qb/K))$ is, up to conjugation, a dihedral group of order dividing $8$. If $\sqrt{5}$ does not belong to $K$ then there exist no exceptional pairs for $K$ with $\ell=5$.
\end{prop}
%
%
\begin{proof}
As remarked in the proof of Theorem~\ref{suth}, for $\ell{=}2$ there exists no exception to the local-global principle.
Take $\ell{=}3$. If $\sqrt{-3}$ is not in $K$ then there exists no exceptional pair since, by Lemma~\ref{comb}, the projective image is a dihedral group of order $2n$ with $n\in \Z_{>1}$ odd (dividing $3-1$). Similarly if $\sqrt{-3}$ belongs to $K$ there exists no exceptional pair since, by Lemma~\ref{comb1}, $\ell{\equiv}1$ mod $4$.
For $\ell{=}5$ we have that if $\sqrt{5}$ is not in $K$ then there exists no exceptional pair by Lemma~\ref{comb}. Moreover, if $\sqrt{5}$ is in $K$ then by Lemma~\ref{comb1} combined with Corollary~\ref{a4a5}, the projective image can only be a dihedral group of order dividing $8$. 
\end{proof}
We will come back to the study of the local\--global principle about $5$-isogenies in Section~\ref{case5}.

\section{Bounds}
In this section we prove statement $(1)$ of the Main Theorem.

\subsection{Image of the inertia}

Let $M$ be a complete field with respect to a discrete valuation $v$, which is normalized, i.e.\ $v(M^\ast) = \Z$. Let $\O_M$ be its ring of integers, $\lambda$ the maximal ideal of $\O_M$ and $k = \O_M/\lambda$ the residue field. We suppose $M$ of characteristic $0$, the residue field $k$ finite of characteristic $\ell{>}0$ and $e = v(\ell)$.
Let $E$ be an elliptic curve having semi-stable reduction over $M$ and let $\mathcal{E}$ be its N\'eron model over $\O_M$. Since $M$ is of characteristic $0$, we know that $E[\ell](\overline{M})$ is an $\F_\ell$-vector space of dimension $2$. 
Let $\overline{\mathcal{E}}$ be the reduction of $\mathcal{E}$ modulo $\lambda$, then $\overline{\mathcal{E}}$ is a group scheme defined over $k$ whose $\ell$-torsion is an $\F_\ell$-vector space with dimension strictly lower than $2$. Hence, the kernel of the reduction map, can be either isomorphic to $\F_\ell$ (ordinary case) or to the whole $E[\ell]$ (supersingular case).

Serre, in \cite[Proposition~$11$, Proposition~$12$ and p.$272$]{serre72}, described all possible shapes of the image of ${I}_\ell$, the inertia subgroup at $\ell$, for the supersingular case and for the ordinary case:
\begin{prop}[(Serre, supersingular case)]
Let $E$ be an elliptic curve over $M$, a complete normalized field with respect to the valuation $v$, and let $e{=}v(\ell)$. Suppose that $E$ has good supersingular reduction at $\ell$.  Then the image of $I_\ell$ through the Galois representation $\rho_{E,\ell} :\Gal(\overline{M}/M)\to \GL_2(\F_\ell)$ associated to $E$ is cyclic of order either ${(\ell^2{-}1)}/{e}$ or $\ell(\ell{-}1)/{e}$. 
\label{se}
\end{prop}
The two cases depend on the action of the tame inertia. 

If the tame inertia acts via powers of the fundamental character of level $2$, and not $1$, it follows that the Newton polygon, with respect to the elliptic curve, is not broken, and that the tame inertia is given by the $e$-power of the fundamental character of level $2$. Hence it has a cyclic image of order ${(\ell^2{-}1)}/{e}$.

If the elliptic curve considered is supersingular, but the tame inertia acts via powers of the fundamental character of level $1$, it follows that the relevant Newton polygon is broken, and there are points in the $\ell$-torsion of the corresponding formal group which have valuation with denominator divisible by $\ell$ (this follows from \cite[p.$272$]{serre72}). So the image of inertia has order $\ell(\ell{-}1)/{e}$.

In the ordinary case, the following proposition holds:
\begin{prop}[(Serre, ordinary case)]\label{se2}
Let $E$ be an elliptic curve over $M$, a complete normalized field with respect to the valuation $v$, and let $e=v(\ell)$. Suppose that $E$ has semistable ordinary reduction at $\ell$. Then the image of $I_\ell$ through the Galois representation $\rho_{E,\ell} :\Gal(\overline{M}/M)\to \GL_2(\F_\ell)$ associated to $E$ is cyclic of order either ${(\ell{-}1)}/{e}$ or  $\ell(\ell{-}1)/{e}$, and it can be represented, after the choice of an appropriate basis, respectively as 
\[\abcd{\ast}{0}{0}{1}\quad \mathrm{ or }\quad \abcd{\ast}{\star}{0}{1},\] 
with $\ast \in \F_\ell^\ast$ and $\star \in \F_\ell$.
\end{prop}

\subsection{ Computation of the bound}

Let $E_\lambda$ be an elliptic curve defined over a complete field $M$ with maximal ideal $\lambda$, residual characteristic $\ell$, and  let 
\begin{equation}
d'= \left\{\begin{array}{ll}
1& \mbox{if}\,\;j(E)\not\equiv 0, 1728 \bmod \lambda,\\
2& \mbox{if}\,\;j(E)\equiv 1728 \bmod \lambda, \ell\geq 5\\
3& \mbox{if}\,\;j(E)\equiv 0 \bmod \lambda,\ell \geq 5,\\
6& \mbox{if}\,\;j(E)\equiv 0 \bmod \lambda, \ell=3,\\
12& \mbox{if}\,\;j(E)\equiv 0 \bmod \lambda, \ell=2.\\
\end{array}\right.
\label{degest}
\end{equation}
Then $E_\lambda$, or a quadratic twist, has semistable reduction over a finite extension of $M$ with degree $d'$, see for instance \cite[pp.$33-52$]{biku} or \cite[Proposition 1 and Th\'eor\`eme 1]{kraus}. 

We now give the main result of this article:
\begin{thm}\label{t4}
Let $( \ell , j(E))$ be an exceptional pair for the number field $K$ of degree $d$ over $\Q$, such that $\sqrt{\left(\frac{-1}{\ell}\right)\ell}\notin K$ and $j(E)\notin \{0, 1728\}$. Then  \[\ell \equiv 3 \bmod 4\;\;\mbox{ and }\;\;7\leq \ell \leq 6d{+}1.\]
\end{thm}
\begin{proof} Since the pair $( \ell , j(E))$ is exceptional, $E$ admits an $\ell$-isogeny locally at a set of primes with density one, $\ell \equiv 3 \bmod 4$ and by Proposition~\ref{1.1}  it admits an $\ell$-isogeny over $L=K(\sqrt{-\ell})$. In particular, $\ell\neq 2,3$ and $\ell\geq 7$.

Let $K_\lambda$ be the completion of $K$ at $\lambda$, a prime above $\ell$, and let $M$ be the smallest extension of $K_\lambda$ over which $E_\lambda:=E \otimes K_\lambda$ acquires semi-stable reduction. After replacing $E$ by a quadratic twist if necessary, we can assume that the extension 
$M/ K_\lambda$ has degree less than or equal to $3$, according to (\ref{degest}), since $\ell\geq 7$. 
Let $\overline{E}$ be the reduction of $E_{\lambda'}:=E\otimes M$ modulo $\lambda'$, for $\lambda'$ the prime above $\lambda$.

We now look at the image of the inertia subgroup under the Galois representation associated to the $\ell$\--torsion of $E$. The Galois representation has image of order dividing $(\ell{-}1)^2$ by Remark~\ref{dim}.

Assume that the reduction $\overline{E}$ is supersingular. The inertia has image isomorphic to a cyclic group of order $(\ell^2{-}1)/m$ or $\ell(\ell{-}1)/m$, where $m$ is less than or equal to $3d$, according to the degree of the extension needed to have semi-stable reduction.
In the first case, i.e.\ when the image of the inertia has order $(\ell^2{-}1)/m$,  it is isomorphic to a non-split torus in $\GL_2(\F_\ell)$ and is also contained in the normalizer of a split Cartan, as stated in Lemma~\ref{comb}. This is impossible unless 
\[(\ell^2{-}1)/m \mid (\ell{-}1)^2.\]
This means that $(\ell{+}1)/m \mid (\ell{-}1)$, so that  $(\ell{+}1)/m \mid 2$, hence $\ell \leq 2m{-}1$. 

Analogously, in the second case we have that  $\ell(\ell{-}1)/m$ divides $(\ell{-}1)^2$, therefore $\ell \leq m\leq 3d$. The pair  $( \ell , j(E))$ is exceptional, so $\ell {\equiv} 3$ mod $4$ and $\ell{ \geq} 7$, hence, $7\leq \ell \leq 2m-1\leq 6d{-}1$.

We have proved that if $\ell > 2m-1$, then $\overline{E}$ is not supersingular, so it is ordinary, since $E_\lambda$ is semistable over $M$. By Proposition~\ref{field}, then the elliptic curve $E$ admits two $\ell$-isogenies over $L=K(\sqrt{-\ell})$, which are conjugate over $L$. 
By Lemma~\ref{comb}, the image $G$ of ${\mathrm{Gal}} (\overline{\Q} /K)$ acting on $E[\ell ]({\overline{K}})$ is a subgroup of the normalizer $N$ of a split Cartan subgroup $C$. From Proposition~\ref{field}, we know that $N/C \simeq \Gal(L /K) \neq \{1\}$, so the image of an inertia subgroup ${I}_\lambda$ at the place $\lambda$ of $K$ is a subgroup of $G$ whose image in $N/C$ is non-trivial. On the other hand, Proposition~\ref{se2} shows that, if $(\ell -1)/m >2$, then ${I}_\lambda$ contains a cyclic subgroup of order larger than or equal to $3$ (for another argument see \cite[p.118]{maz1}). It follows that ${I}_\lambda$ contains a non-trivial Cartan subgroup (of shape $\left\{\abcd{a}{0}{0}{b}: a, b \in \Gamma \right\}$ for a certain non trivial subgroup $\Gamma$ of $\F_\ell$), even after restriction of the scalars to ${\mathrm{Gal}} (\overline{M} /M)$. This is a contradiction with~Proposition~\ref{se2}, as the latter says that the restriction of ${ I}_\lambda$ to ${\mathrm{Gal}} (\overline{M} /M)$ is a semi-Cartan subgroup (or a Borel).
Hence, $(\ell -1)/m \leq 2$ so we have $\ell \equiv 3 \bmod 4$ and $7\leq \ell \leq 2m{+}1\leq 6d{+}1$.
\end{proof}

\begin{remark} It is clear that Theorem~\ref{t4} implies the result obtained by Sutherland in the case $K=\Q$. 
\end{remark}

The previous Theorem, combined with Remark~\ref{sqr}, implies point $(1)$ of the Introduction's Main Theorem, namely:
\begin{cor}\label{maincor} 
Let $( \ell , j(E))$ be an exceptional pair for the number field $K$ of degree $d$ over $\Q$ and discriminant $\Delta$. Assume  $j(E)\notin \{0, 1728\}$.  Then \[\ell \leq \max\left\{|\Delta|, 6d{+}1\right\}.\]
\end{cor}
\begin{proof} If $( \ell , j(E))$ is an exceptional pair for the number field $K$ then we distinguish two cases according to the projective image being contained or not in $\SL_2(\F_\ell) /\left\{\pm 1\right\}$. This corresponds to a condition about $\sqrt{\left(\frac{-1}{\ell}\right)\ell}$ belonging to $K$ or not. If  $\sqrt{\left(\frac{-1}{\ell}\right)\ell}\notin K$ we can apply Theorem~\ref{t4} and conclude that $7\leq \ell \leq 6d{+}1$. While, if $\sqrt{\left(\frac{-1}{\ell}\right)\ell}\in K$, then $\ell$ divides $|\Delta|$.\end{proof}

\section{Finiteness of the exceptional pairs}
Given a number field $K$ of degree $d$ over $\Q$ and discriminant $\Delta$, the local-global principles for $\ell$-isogenies holds whenever  
$\ell > \ell_k:=\max\left\{|\Delta|, 6d{+}1\right\}$ or $\ell=2$ or $3$ by Corollary~\ref{maincor} and Proposition~\ref{small}. In this section we analyse what happens for primes  smaller than the bound obtained. In particular we will prove that the local-global principle about $\ell$-isogenies for elliptic curves over number fields admits only a finite number of exceptions if $\ell>7$. We will also study the behaviour of the local-global principle at $5$ and $7$.\\

Let $K$ be a number field and let $C/K$ be a projective smooth curve defined over $K$ and of genus $g$. Our arguments will rely on the classical trichotomy between curves of genus $0$, $1$ and higher. When the genus is $0$, the curve is isomorphic to~$\mathbb{P}^1_{\overline{K}}$ over an algebraic closure of $K$ and therefore $C(K)$, the set of $K$-rational points, is either empty or infinite. If the genus of $C$ is $1$ and $C(K)$ contains at least one point over $K$ then $C/K$ is an elliptic curve over $K$ and the Mordell-Weil theorem shows that $C(K)$ is a finitely generated abelian group: $C(K)\cong T \oplus \Z^r$, where $T$ is the torsion subgroup and $r$ is a non-negative integer called the rank of the elliptic curve, whereas if  $g\geq2$, Faltings Theorem states that the set of $K$-rational points is finite. \\

In this section we will recall some theory of modular curves.  

Let $\ell\geq 5$ be a prime number and let $\Z[\zeta_\ell]$ be the subring of $\C$ generated by a root of unity of order $\ell$. The modular curve $X(\ell)$ is the compactified fine moduli space which classify pairs $(E,\alpha)$, where $E$ is a generalized elliptic curve over a scheme $S$ over 
$\spec(\Z[1/\ell, \zeta_\ell])$ and $\alpha:(\Z/\ell\Z)_S^2\stackrel{\sim}{\rightarrow}E[\ell]$ is an isomorphism of group schemes over $S$ which is a full level $\ell$-structure, up to isomorphism of pairs,  i.e.\ isomorphisms of elliptic curves that preserve the level structure. 
A full level $\ell$-structure on a generalized elliptic curve $E$ over $S$ is a pair of points $(P_1, P_2)$, satisfying 
$P_1, P_2 \in E[\ell]$ and $e_\ell(P_1, P_2) = \zeta_\ell$ where $e_\ell$ is the Weil pairing on $E[\ell]$. Let us recall that a full level $\ell$-structure induces a symplectic pairing on $(\Z/\ell\Z)^2$ via $\left\langle (1,0),(0,1)\right\rangle=\zeta_\ell$.  For more details, see \cite{katzmaz} or \cite{gross}.\\

Let $G$ be a subgroup of $\GL_2(\Z/\ell\Z)$. We will denote as $X_G(\ell):=G\backslash X(\ell)$ the quotient of the modular curve $X(\ell)$ by the action of $G$ on the full level $\ell$-structure. The modular curve  $X_G(\ell)$ has a geometrically irreducible model over $\Q(\zeta_\ell)^{\det(G)}$, the subfield of  $\Q(\zeta_\ell)$ invariant under the action of ${\det(G)}$, see 
\cite[pp.$115-116$]{maz1} or \cite[IV, $3.20.4$]{derap}. 

In particular, if $G$ is the Borel subgroup, $X_G (\ell)$ is the modular curve $X_0 (\ell)$  over $\Q$. This modular curve parametrizes elliptic curves with a cyclic $\ell$-isogeny, that is, pairs $(E,C)$, where $E$ is a generalized elliptic curve and $C$ is the kernel of a cyclic $\ell$-isogeny, up to isomorphism. 

If $G$ is a split Cartan subgroup (respectively, the normalizer of a split Cartan subgroup) we will denote the modular curve $X_G(\ell):=X_{\spc}(\ell)$  (respectively, $X_{\spl}(\ell)$). The curve $X_{\spc}(\ell)$ (respectively, $X_{\spl}(\ell)$) pa\-ram\-e\-trizes elliptic curves endowed with an ordered
(respectively, unordered) pair of independent cyclic $\ell$-isogenies. 

Following \cite{maz1}, we will denote as $X_{A_4}(\ell)$ (respectively $X_{S_4}(\ell)$, $X_{A_5}(\ell)$) the modular curves obtained taking as 
$G\subset \GL_2(\Z/\ell\Z)$ the inverse image of $A_4\subset \PGL_2(\Z/\ell\Z)$ (respectively $S_4, A_5\subset \PGL_2(\Z/\ell\Z)$). Let us remark that exceptional projective images $A_4, S_4$ and $A_5$ can occur only for particular values of $\ell$, see \cite[section $2.5$, $2.6$]{serre72}. The modular curves $X_{A_4}(\ell)$ and $X_{A_5}(\ell)$ have geometrically irreducible models over the quadratic subfield of $\Q(\zeta_\ell)$. The same holds for $X_{S_4}(\ell)$ if $\ell\not\equiv \pm3$ mod $8$, otherwise the model is defined over $\Q$. 
\begin{remark}
Let $E/K$ be an elliptic curve that arises in an exceptional pair $(\ell , j(E))$ for the number field $K$. Let us suppose that the projective image of $\rho_{E,\ell}$ is dihedral. Hence, $(E,\rho_{E,\ell})$ corresponds to a $K$-rational point in $X_{\spl}(\ell)$ by Lemma~\ref{comb} and Corollary~\ref{spC}. Moreover, $E[\ell](\overline{K})$ contains two conjugate lines $L_1$ and $L_2$ over $L/K$, where $L/K$ is quadratic (Propositions~\ref{1.1} and \ref{1.11}). These lines correspond to the isogenies $\alpha: E\to E/L_1$ and $\beta: E \to E/L_2$ defined over $L$. Hence, they give a pair of $L$-rational points (taking respectively $\alpha\beta^\vee$ and $\beta\alpha^\vee$ as isogeny structure) on $X_0(\ell^2)$ which are conjugate by the Fricke involution $w_{\ell^2}$; for a definition see \cite[section $2$]{Parent}. Let us recall that there exists an isomorphism defined over $\Q$ between $X_0(\ell^2)$ and $X_{\spc}(\ell)$.
\label{cov}
\end{remark}

\begin{remark}\label{mod12}
If $(\ell , j(E))$ is an exceptional pair for the number field $K$ and $\sqrt{\left(\frac{-1}{\ell}\right)\ell}\notin K$ then the prime $\ell$ is congruent to $3 \bmod 4$ and hence to $7$ or $11 \bmod 12$, while if $\sqrt{\left(\frac{-1}{\ell}\right)\ell}\in K$ then the prime $\ell$ is congruent to $1 \bmod 4$ and hence to  $1$ or $5 \bmod 12$. 
\end{remark}

\subsection{The case $11 \leq \ell \leq \ell_K$}

\begin{thm}\label{fin}
Let $K$ be a  number field. If $\ell$ is a prime greater than $7$, then the number of exceptional pairs $( \ell , j(E))$ for $K$ is finite. 
\end{thm}
\begin{proof} Given an exceptional pair $( \ell , j(E))$ for the number field $K$ it corresponds to a $K$-rational point on one of the following modular curves:  $X_{\spl}(\ell)$, $X_{A_4}(\ell)$, $X_{S_4}(\ell)$ or $X_{A_5}(\ell)$, by Lemmas~\ref{comb} and~\ref{comb1}. Let us analyse each possible case. 

Let us recall that the genus of $X_{\spl}(\ell)$ is given by the following formula, see \cite[p. 117]{maz1}:
\[g(X_{\spl}(\ell))=\frac{1}{24}\left(\ell^2-8\ell+11-4\left(\frac{-3}{\ell}\right)\right).\]
Hence, if $\ell \equiv 1$ or $7 \bmod 12$, then $g(X_{\spl}(\ell))=\frac{1}{24}(\ell^2-8\ell+7)$. Otherwise, if $\ell \equiv 5$ or $11 \bmod 12$, then $g(X_{\spl}(\ell))=\frac{1}{24}(\ell^2-8\ell+15)$. Therefore, the modular curve $X_{\spl}(\ell)$ has genus greater than or equal to $2$ for $\ell\geq 11$, and has only finitely many $K$-rational points by Faltings Theorem.

Let us now study the modular curves $X_{A_4}(\ell)$, $X_{S_4}(\ell)$ and $X_{A_5}(\ell)$. The genus of these modular curves is given by the following formulae, see \cite[section $2$]{coha}:
\begin{align*}
&g(X_{A_4}(\ell))=\frac{1}{288}(\ell^3-6\ell^2-51\ell+294+18\epsilon_2+32\epsilon_3)\\
&g(X_{S_4}(\ell))=\frac{1}{576}(\ell^3-6\ell^2-87\ell+582+54\epsilon_2+32\epsilon_3)\\
&g(X_{A_5}(\ell))=\frac{1}{1440}(\ell^3-6\ell^2-171\ell+1446+90\epsilon_2+80\epsilon_3)
\end{align*}
where $\epsilon_2$ is equal to $1$ if $\ell \equiv 1 \bmod 4$ and to $-1$ if $\ell \equiv 3 \bmod 4$, and $\epsilon_3$ is equal to $1$ if $\ell \equiv 1 \bmod 3$ and to $-1$ if $\ell \equiv -1 \bmod 3$. 
We stress again that these exceptional cases occur only for certain values of $\ell$, see \cite[section $2.5$, $2.6$]{serre72}, and the formulae given will not be integral for general values of $\ell$, as already noticed in \cite[p.$3072$]{coha}. 
By Corollary~\ref{a4a5}, if an exceptional pair has projective image isomorphic to $A_4$ then 
$\ell \equiv 1 \bmod 12$ and we note that the genus of $X_{A_4}(\ell)$ is greater than $2$ for all $\ell\geq 13$. Similarly, for projective image isomorphic to $S_4$ or to $A_5$ the genus of the respective modular curves is larger than $2$ for primes satisfying the appropriate congruence.
\end{proof}

\subsection{The case $\ell =5$}
\label{case5}

Now we will study the local-global principle for $5$-isogenies. In order to do so it will be relevant to recall the structure of $X(5)$ at the cusps. 

The modular interpretation of $X(5)(\overline{\Q})$ associates to each cusp a N\'eron polygon $\mathcal{P}$ with $5$ sides. The N\'eron polygon is endowed with the structure of a generalized elliptic curve enhanced with a basis of $\mathcal{P}[5]\cong \mu_5 \times \Z/5\Z$, where $\mu_5$ is the set of $5$-th roots of unity, up to automorphisms of $\mathcal{P}$:
\begin{eqnarray*}
(\{\pm1\}\times \mu_5)\times \mathcal{P}[5] &\rightarrow & \mathcal{P}[5]\\
\left(\abcd{\epsilon}{\alpha}{0}{\epsilon} ,{w\choose {j}}\right) &\mapsto & {{w^\epsilon \alpha^j}\choose {\epsilon j}}
\end{eqnarray*}
where $\epsilon{\in}\{\pm1\}$ and $\alpha, w\in \mu_5$. The set of cusps of  $X(5)(\overline{\Q})$ is a Galois set with an action of $\GL_2(\Z/5\Z)$. The modular interpretation of $X_G(5)$ associates to each cusp an orbit of the enhanced N\'eron polygon under the action of $G$.

The local-global principle for $5$-isogenies is related to $V_4$, the Klein $4$-group. Let us recall that there is a unique non-trivial $2$-dimensional irreducible projective representation $\tau$ of $V_4$ in $\PGL_2(\F_5)$ and, up to conjugation, this representation is given by the image in $\PGL_2(\F_5)$ of the set:
\begin{equation}\left\{ {\begin{pmatrix} 1 & 0 \\ 0 & 1 \\\end{pmatrix}}, {\begin{pmatrix} 0 & -1 \\ 1 & 0 \\\end{pmatrix}},
{\begin{pmatrix} 1 & 0 \\ 0 & -1 \\\end{pmatrix}}, {\begin{pmatrix} 0 & 1 \\ 1 & 0 \\\end{pmatrix}}\right\}.
\label{set}
\end{equation}
For a prime $\ell\geq 5$, we denote as $X_{V_4}(\ell)$ the modular curves $X_G(\ell)$ obtained by taking as $G\subset \GL_2(\Z/\ell\Z)$ the inverse image of $V_4\subset \PGL_2(\Z/\ell\Z)$.

\begin{prop} \label{cusp}
Over $\spec(\Q(\sqrt{5}))$ the modular curve $X_{V_4}(5)$ is isomorphic to $\mathbb{P}^1$.
\end{prop}

\begin{proof}The genus of $X(5)$ over $\Q(\zeta_5)$ is $0$.  
The field of constants of  $X_{V_4}(5)$ is  $\Q(\zeta_5)^{\det(G)}$ where $G$ is the inverse image of $V_4$ in $\GL_2(\F_5)$. Since $V_4 \subset \SL_2(\F_5)/\{\pm1\}$ and $\F_5^\ast\subset G$, then $\det(G)=(\F_5^\ast)^2$. This means that $X_{V_4}(5)$ is geometrically irreducible over $\Q(\sqrt{5})$ and its genus is $0$. 

The set of cusps of $X(5)(\overline{\Q})$ is in $1{-}1$ correspondence with the quotient of the group of isomorphisms as $\F_5$-vector spaces between $\F_5^2$ and $\mu_5\times\F_5$ by the action of $\{\pm1\}\times \mu_5$. To show that over $\spec(\Q(\sqrt{5}))$ the modular curve $X_{V_4}(5)$ is isomorphic to $\mathbb{P}^1$ it is enough to show that the set of  $\Q(\sqrt{5})$-rational points of $X_{V_4}(5)$ is non-empty.

Let $\phi_{\zeta_5}:\F_5^2\stackrel{\sim}{\rightarrow}\mu_5\times\F_5$ be the isomorphism given by $\phi_{\zeta_5}((1,0))=(\zeta_5,0)$, $\phi_{\zeta_5}((0,1))=(1,1)$. The orbit of the cusp corresponding to $\phi_{\zeta_5}$ under the action of $\{\pm1\}\times \mu_5$ is the following set:
\begin{eqnarray*} 
&&\{((\zeta_5,0),(1,1)),((\zeta_5^{-1},0),(1,-1)),((\zeta_5,0),(\zeta_5,1)),((\zeta_5^{-1},0),(\zeta_5^{-1},-1)),\\ &&((\zeta_5,0),(\zeta_5^2,1)),((\zeta_5^{-1},0),(\zeta_5^{-2},-1)),((\zeta_5,0),(\zeta_5^{-2},1)),\\
&&((\zeta_5^{-1},0),(\zeta_5^{2},-1)),((\zeta_5,0),(\zeta_5^{-1},1)),((\zeta_5^{-1},0),(\zeta_5,-1))\}.
\end{eqnarray*} 

On the set of cusps we have a Galois action by $\Gal(\Q(\zeta_5)/\Q)$. In particular, acting with $-1\in\Gal(\Q(\zeta_5)/\Q)$ on the cusp $((\zeta_5,0),(1,1))$ we obtain the cusp $((\zeta_5^{-1},0),(1,1))$ which does not define the same cusp on $X(5)({\overline{\Q}})$. 

However, the action of $-1\in\Gal(\Q(\zeta_5)/\Q)$ on the class of the cusp $((\zeta_5,0),(1,1))$ preserve the orbit of the cusp under $G$: in fact $G$, the inverse image of $V_4$ in $\GL_2(\F_5)$, is  the group 
\[\left\{\abcd{x}{0}{0}{\pm x},\abcd{0}{\pm x}{x}{0} :x\in\F_5^\ast\right\},\] and under the action of $\abcd{-1}{0}{0}{1}\in G$ the cusp $((\zeta_5^{-1},0),(1,1))$ of $X(5)(\overline{\Q})$ is mapped to $((\zeta_5,0),(1,1))$. 

It follows that the cusp $((\zeta_5,0),(1,1))$ in $X_{V_4}(5)(\overline{\Q})$ is stable under $\Gal(\Q(\zeta_5)/\Q(\sqrt{5}))$. This implies that $X_{V_4}(5)(\Q(\sqrt{5}))$ is non-empty.
\end{proof}

\begin{remark}The argument given above shows that  $X_{V_4}(5)(\Q(\sqrt{5}))$  is non-empty (and therefore infinite) by proving that there exist a $\Q(\sqrt{5})$-rational cusp. As an alternative, it is enough to  exhibit a $\Q(\sqrt{5})$-rational point on $X_{V_4}(5)$. Over $\Q$, the mod-$5$ Galois image of the elliptic curve with LMFDB label $608.$e$1$\footnote{see \cite{lmfdb:ellQ}, Cremona label $608$b$1$}:
\[E: y^2 = x^3 - 56x + 4848\]
is equal to the normalizer of the split Cartan in $\GL_2(\F_5)$:  up to conjugacy, the normalizer of the split Cartan is the only subgroup of $\GL_2(\F_5)$ with order $32$ and $[\Q(E[5]) : \Q] = 32$. Changing the field of definition of $E$ to  $\Q(\sqrt{5})$ reduces the Galois image to the index-$2$ subgroup of the normalizer of the split Cartan with square determinants. The image of this subgroup in $\PGL_2(\F_5)$ is isomorphic to $V_4$, therefore $j(E)$ corresponds to a $\Q(\sqrt{5})$-rational point on $X_{V_4}(5)$.
 Through the LMFDB database, see \cite{lmfdb} and in particular the search engine  \cite{lmfdb:ellQ}, it is possible to find lots of examples of elliptic curves whose mod-$5$ Galois image is equal to the normalizer of the split Cartan in $\GL_2(\F_5)$, such as the elliptic curves with  LMFDB labels $121.$b$2$ and $1216.$h$1$.
\end{remark}
\begin{cor} 
There exist infinitely many exceptional pairs $(5 , j(E))$ for the number field $K$ if and only if $\sqrt{5}$ belongs to $K$.
\label{five}
\end{cor}
\begin{proof} By Proposition~\ref{small}, there is an exceptional pair $(5 , j(E))$ for the number field $K$ only if $\sqrt{5}$ belongs to $K$.

 If $(5 , j(E))$ is an exceptional pair  for the number field $K$ then the projective image of the Galois representation associated to the elliptic curve $E$ over $K$ is a dihedral group of order dividing $8$ (Lemma~\ref{comb1} combined with Corollary~\ref{a4a5}). 
If $\sqrt{5}$ belongs to $K$, each of the infinitely many (Proposition~\ref{cusp}) non-cuspidal points on  $X_{V_4}(5)(K)$ corresponds to the isomorphism class of an elliptic curve $E$ whose Galois image in $\PGL2(\F_5)$ is isomorphic to $V_4$. Every group $G\subset \GL_2(\F_5)$ with image $V_4$ in $\PGL_2(\F_5)$ has the property that each of its elements fixes a line in $\F_{25}$ but $G$ does not, for
example, no line in $\F_{25}$ is fixed by both the second and third matrices in the set (\ref{set}). Therefore all the pairs $(5, j(E))$ are exceptional. \end{proof}
%
%
%
%

\begin{remark} In \cite{barcre}, it  is given an explicit formulafor the map from $X_{V_4}(5)$ to $X(1)$ and, hence, an explicit parametrization of the $j$\--invariants of the exceptional pairs $(5, j(E))$.
\end{remark}

\subsection{The case $\ell =7$}

The local-global principle for $7$-isogenies leads us to a dichotomy between a finite and an infinite number of counterexamples according to the rank of a specific elliptic curve that we call the Elkies-Sutherland curve: 
 
\begin{prop}\label{7}
If $\ell=7$ then the number of exceptional pairs $( 7 , j(E))$ for a number field $K$, is finite or infinite, depending whether the rank of the elliptic curve 
\[E': y^2=x^3-1715x+33614\] is zero or non-zero, respectively.
\end{prop}
\begin{proof} If $\sqrt{-7}\in K$ then by Lemma~\ref{comb1} there is no exceptional pair. Let us suppose that $\sqrt{-7}$ is not in $K$. As shown by Sutherland in 
\cite[Section $3$]{suth} and explained in Remark~\ref{cov}, the modular curve of interest is the twist of $X_0(49)$ by $\Gal(K(\sqrt{-7})/K)$ with respect to $w_{49}$, the Fricke involution on $X_0(49)$. 
As explained in \cite[Section $3$]{suth}, computations of Elkies imply that this modular curve is isomorphic to $E'$, thus each $K(\sqrt{-7})$\--rational point on $E'$ gives rise to an exceptional pair $( 7 , j(E))$. Explicitly, if the $K(\sqrt{-7})$\--rational point of $E'$ has coordinates $(u,v)$, let $t=(7u-v+343)/2v$, then the $j$-invariant of the isomorphism class of elliptic curves which are exceptional for the local-global principle for $7$-isogenies is equal to 
\[\frac{ -(t-3)^3(t-2)(t^2+t-5)^3(t^2+t+2)^3(t^4-3t^3+2t^2+3t+1)^3}{(t^3-2t^2-t+1)^7}.\]
Hence, if the rank of $E'$ over $K$ is positive there are infinitely many counterexamples to the local-global principle about $7$-isogenies, while if the rank is $0$ there are only finitely many.
\end{proof}

\begin{remark} As shown by Sutherland in \cite[Section $3$]{suth}, over $\Q(i)$ the curve $E$ has positive rank, so there are infinitely many counterexamples over this field to the local-global principle for $7$-isogenies.
\end{remark}

The proof of our Main Theorem is now complete.

\section{Complex multiplication}
Sutherland in \cite{suth} proved that an exceptional pair $(\ell , j(E))$ over $\Q$ cannot have complex multiplication: for $\ell>7$ we refer to the proof of \cite[Theorem $2$]{suth} and for $\ell=7$ we refer to the direct computations in \cite[Section $3$]{suth}. Here we study the same problem for $K$ a number field.

\begin{lem} \label{cm}Let $K$ be a number field of degree $d$ over $\Q$ and let $E/K$ be an elliptic curve over $K$ with $j(E)\notin \{0, 1728\}$. Let $(\ell , j(E))$ be an exceptional pair for $K$. If $\ell> 2d{+}1$ then $E$ does not have complex multiplication. 

\end{lem}
\begin{proof}  
Assume that $E$ has complex multiplication by a quadratic order $\O$. This means that the Galois representation $\rho_{E,\ell}$ has image included in a Borel, when $\ell$ ramifies in $\O$, or is projectively dihedral (split or nonsplit, according to $\ell$ being split or inert in $\O$), see \cite[Th\'eor\`eme $5$]{serre66}. 
The Borel case is clearly not possible.
By Proposition~\ref{1.1}, there is an $\ell$-isogenous elliptic curve $E'$ that is defined over a quadratic extension $L$ of $K$, but not over $K$, since we are in an exceptional case. Since $E'$ is isogenous to $E$ over $L$, it also must have complex multiplication by an order $\O'$. Since $E$ and $E'$ are $\ell$-isogenous, by \cite[Theorem $7.24$]{cox}, the ratio between the class numbers  $h(\O')$ and $h(\O)$ satisfies
\[\frac{h(\O')}{h(\O)}=\frac{1}{\left[\O^\ast : \O'^\ast\right]}\left(\ell  - \left(\frac{\disc(\O)}{\ell}\right)\right)\geq (\ell{-}1),\]
since $j(E)\notin \{0, 1728\}$ and therefore $\left[\O^\ast : \O'^\ast\right]=1$. In particular, if ${h(\O')}={h(\O)}$ we have a contradiction. Hence assume ${h(\O')}>{h(\O)}$.
Since $E'$ is defined over $L$ and $E$ is defined over $K$, we know that~$\Q (j_E )\subseteq K$ and~$\Q (j_{E'} )\subseteq L$. Then the ratio of the class numbers  $h(\O')$ and $h(\O)$ satisfies: 
\[\frac{h(\O')}{h(\O)}\leq \left[L:\Q\right]=2d.\]
Hence, if $\ell> 2d+1$, we  have a contradiction between the lower and the upper bound, so $E$ cannot have complex multiplication.
\end{proof}
\section*{Acknowledgements}
I would like to thank my advisor Pierre Parent, for his advice, his guidance and patience throughout the many readings and corrections of this article. I also am very grateful to my advisor Bas Edixhoven, for all the fruitful discussions about this topic. I would also like to thank the referee for the thorough, constructive and helpful comments and suggestions on the manuscript.

\bibliographystyle{alpha}
\footnotesize
\bibliography{biblio}

\begin{flushleft}
\begin{table}[h]
\footnotesize{\begin{tabular}{l}
Mathematics Institute, 
Zeeman Building\\
University of Warwick\\
Coventry CV4 7AL\\
United Kingdom
\end{tabular}}
\end{table}
\textit{\footnotesize{E-mail address: {\url{ s.anni@warwick.ac.uk}}}}
\end{flushleft}

\end{document}